\documentclass[11pt]{article}
\usepackage[a4paper, tmargin=3cm, bmargin=3.0cm, lmargin=2.5cm, rmargin=2.5cm, textheight=24cm, textwidth=16cm]{geometry}
\usepackage{setspace}
\usepackage{amsfonts}
\usepackage{epsfig}
\usepackage{latexsym}
\usepackage{amsthm}
\usepackage{amsmath}
\usepackage{amssymb}
\usepackage{array}
\usepackage{enumerate}
\usepackage{enumitem}
\usepackage{graphicx}
\usepackage{tikz} 
\usetikzlibrary{shapes.geometric} 
\usepackage{amsmath} 
\usepackage{lmodern} 
\usepackage{breqn}
\usepackage{thmtools, thm-restate}

\newtheorem{theorem}{Theorem}[section]

\newtheorem{corollary}{Corollary}[section]

\newlist{notes}{enumerate}{1}
\setlist[notes]{label=Note: ,leftmargin=*}

\title{On the Spectral Analysis of the Superpower Graph of the Direct Product of Dihedral Groups}
\author{Basit Auyoob Mir, Fouzul Atik
 \\ \small Department of Mathematics, SRM University-AP, Andhra Pradesh 522240, India.\\ \small e-mail: mirbasit553@gmail.com, fouzulatik@gmail.com, }

\date{}

\begin{document}
\maketitle{}
\begin{abstract}
The superpower graph of a finite group $G$, or $\mathcal{S}_G$, is an undirected simple graph whose vertices are the elements of the group $G$, and two distinct vertices $a,b\in G$ are adjacent if and only if the order of one vertex divides the order of the other vertex, which means that either $o(a)|o(b)$ or $o(b)|o(a)$. In this paper, we have investigated the $A_\alpha$-adjacency spectral properties of the superpower graph of the direct product $D_p\times D_p$, where $D_p$ is a dihedral group for $p$ being prime. Also, we have determined its Laplacian and signless Laplacian spectrum by giving different values to $\alpha$; furthermore, we delved into its superpower graph and deduced the $A_\alpha$- adjacency spectrum of the superpower graph of $D_p\times D_p$ and $D_{p^m}$ for $p$ being an odd prime.
\end{abstract}


\section{Introduction}
Throughout the paper, all the groups and graphs taken here are assumed to be finite, and a graph means a simple undirected graph. For this paper, just basic graph knowledge will be required. Any book on graph theory, for instance, will have them \cite{West}. Our group theory notations are taken from \cite{Rose} and we refer to \cite{Cvetkovic, Doob} for the
algebraic graph theory concepts and notations. Graph theory is now a practical method in the understanding and description of relations, whether it be the social networks or within biological systems, and also on theoretical physics. In this rather vast area, graphs that arise in the setting of algebraic structures and, especially, groups, have attracted much attention from many researchers, and as a result, several structures, collectively termed as the graphs of group theory, are created, including Cayley graphs, commuting graphs, generating graphs, power graphs, and superpower graphs. These represent a geometric viewpoint on properties of underlying groups, and spectral graph theory techniques can be used to give a more algebraic insight. The natural connection between algebraic structures and the associated graphs has the tendency to produce complicated relations, with the spectral characteristics of the graph revealing crucial facts about the structure of the group itself.

Among the various graph constructions on groups, \emph{superpower graphs} represent a relatively new and intriguing class. The superpower graphs of finite groups are a quite recent development in the domain of graphs from groups, and they were first introduced by Hamzeh and Ashrafi, who
they call the order superpower graph $\mathcal{S}_G$ of the power graph $\mathcal{G}_G$ of a finite group, in 2018 \cite{Hamzeh2018}. A superpower graph, represented as $\mathcal{S}_G$, is defined as the graph in which vertices are the elements of the group $G$, and two distinct vertices $a,b\in G$ are adjacent if and only if the order of one vertex $a$ divides the order of the other vertex $b$ or the order of the vertex $b$ divides the order of vertex $a$. Hamzeh et al. \cite{Hamzeh2017} call this graph the main supergraph, and they investigated its full automorphism group. Recently, Hamzeh and Ashrafi explored some characteristics of the order supergraph of a group, and precisely, they showed that $\mathcal{S}_G=\mathcal{G}_G$ if
and only if G is cyclic \cite{Hamzeh2018}. They also investigated the 2-connectedness, Eulerianness, and Hamiltonianity of an order supergraph \cite{Hamzeh2019}.

With these motivations, we consider the superpower graphs of any non-abelian finite group $G$. Like the dihedral groups, denoted as $D_n$, are simple non-abelian groups; they can be thought of as symmetries of regular n-gons. Their complicated algebraic nature and their extensive use in diverse areas of mathematics and physics made them the best objects of detailed investigation in the context of graph theory. The direct product of groups, in their turn, gives a way of building more elaborate algebraic structures from more understandable ones, in which individual properties of a group can be combined and interact with one another in a bigger whole. Interpretations of spectral properties of superpower graphs of direct products of the dihedral graphs, e.g., of the type $D_p \times D_p$, can be insightful into how these graph operations behave over more complex group structures.

Later, people give sharp bounds
for the vertex connectivity of superpower graphs $\mathcal{S}_{D_{n}}$ and $\mathcal{S}_{Q_{n}}$ of dihedral groups $D_{n}$ and dicyclic group $\mathcal{Q}_{n}$ \cite{Kuma}. This paper significantly contributes to the understanding of the adjacency and Laplacian spectral properties of superpower graphs of some finite groups, particularly direct products of dihedral groups. We establish a full spectrum analysis by looking at the adjacency matrix, the Laplacian matrix, and the $A_\alpha$-matrix. According to Nikiforov's proposal in \cite{Nikiforov2017}, the convex combinations of $A(G)$ and $D(G)$ defined by $A_\alpha(G) = \alpha D(G) + (1-\alpha)A(G)$ where $A(G)$ is the adjacency matrix and $D(G)$ is the diagonal matrix of vertex degrees and $\alpha \in [0,1]$, provide the family of matrices between the adjacency matrix ($\alpha=0$) to the signless Laplacian matrix ($\alpha=1$) \cite{Nikiforov2017}. By examining $A_\alpha(G)$, one can gain a deeper comprehension of the spectral behavior of a graph. For more recent papers on the spectral properties of $A_\alpha(G)$, we refer the reader to \cite{Wang, NikiforovTrees, Pirzada2021, Dan} and the references therein.

In particular, the characteristic polynomial of the adjacency matrix of the superpower graph of $D_p\times D_p$, that is, the characteristic polynomial of $A(\mathcal{S}_{D_p\times D_p})$, has been found in this study. Additionally, for $G=D_p \times D_p$, we expanded it to calculate the Laplacian spectrum of $\mathcal{S}_G$, for $G=D_p\times D_p$. The $A_\alpha$-adjacency spectrum of $\mathcal{S}_{D_{p^k}}$, where $p$ is an odd prime, was also examined. The $A_\alpha$-adjacency spectrum of $\mathcal{S}_{D_p \times D_p}$ was also calculated. Also, since the $A_\alpha(G)$ is real and symmetric, we can arrange its eigenvalues in decreasing order as $\lambda_1^{\alpha} \geq \lambda_2^{\alpha} \geq \cdots \geq \lambda_n^{\alpha}$. In this paper, we have represented the spectrum of eigenvalues with their corresponding multiplicities $m_1,m_2,\cdots,m_k$ as  $\text{Spec}(A_\alpha(\mathcal{S}_{D_p\times D_p})) =
\left(
\begin{array}{cccc}
\lambda_1 & \lambda_2 & \cdots & \lambda_n \\
m_1 & m_2 & \cdots & m_k 
\end{array}
\right)$, where $k\leq n$.
The paper is structured as follows: Basic definitions and preliminaries that are utilized in the major results are included in Section 2. The primary findings were provided in Section 3. Section $4$ contains the $A_\alpha$-Adjacency of Superpower Graph of $D_{p^k}$.  The conclusion of the paper is found in section 5.

\section{Preliminary Results}
The objective of this section is to provide certain concepts and results from group theory and graph theory with the aim of achieving the objective of this work.  In order to develop notations, we rewrite the standard definitions and conclusions from \cite{Bond} for graph theory and \cite{Gall} for group theory.  All of the groups in the study are finite.  For a group $G$, the order of an element $x$ is represented by $o(x)$.  The dihedral group of order $2n$ is also denoted as $D_{n}$. It is a non-commutative group formed by two elements $\langle a, b \rangle$ such that $a$ and $b$ meet the following properties: (i) $o(a)=n$, $o(b)=2$ (ii) $ba=a^{-1}b=a^{n-1}b$. For a graph $G$, its diagonal matrix of degrees is $D(G)$, and its adjacency matrix is $A(G)$.  Similar to the \emph{signless Laplacian} $Q(G)= A(G) + D(G)$, the hybrid study of $A(G)$ and $D(G)$ was proposed by Cvetković in \cite{Cvetkovic} and was subsequently thoroughly investigated.  The study of the signless Laplacian matrix $Q(G)$ has demonstrated that it is a remarkable matrix with a wide range of diversity.  However, $Q(G)$ is just the sum of $A(G)$ and $D(G)$, and studies on $Q(G)$ have demonstrated the differences and similarities between $Q(G)$ and $A(G)$.

The convex linear combination of the matrices \( A(G) \) and \( D(G) \) is naturally considered to investigate their effect on the spectral features of \( Q(G) \).  \emph{generalized adjacency matrix} for \( 0 \leq \alpha \leq 1 \) is $A_\alpha(G) = \alpha D(G) + (1 - \alpha) A(G)$.  With efficiency, the previously mentioned equation interpolates between the degree and adjacency matrices.  We have, in particular, $A_0(G) = A(G), \quad A_1(G) = D(G), \quad \text{and} \quad A_{1/2}(G) = \frac{1}{2} Q(G) $. As \( A_\alpha(G) \) is a real symmetric matrix, all of its eigenvalues are real and can be arranged as follows: $\lambda_1(A_\alpha(G)) \geq \lambda_2(A_\alpha(G)) \geq \cdots \geq \lambda_n(A_\alpha(G))$, where $\lambda_1(A_\alpha(G))$ is known as the \emph{generalized adjacency spectral radius} of $G$.

\begin{theorem}\label{upperblockmat}\emph{\cite{Horn}} Let \( M \) be a block upper triangular matrix of the form

\[
M =
\begin{bmatrix}
M_{11} & M_{12} & \cdots & M_{1k} \\
0 & M_{22} & \cdots & M_{2k} \\
\vdots & \vdots & \ddots & \vdots \\
0 & 0 & \cdots & M_{kk}
\end{bmatrix}
\]

where each \( M_{ii} \) is a square matrix. Then, the determinant of \( M \) is given by

\[
\det(M) = \det(M_{11}) \det(M_{22}) \cdots \det(M_{kk}).
\]
\end{theorem}

\section{Main Results}
We start this section with the spectral properties of the superpower graph of the direct product of two dihedral groups. More precisely, we have the group $G = D_p \times D_p$, where $p$ is an odd prime. The examination of the spectrum of the superpower graph gives information about the algebraic and combinatorial structure of the group. In the next theorem, we calculate the $A_\alpha$-adjacency spectrum of $\mathcal{S}_{D_p \times D_p}$ explicitly.

\begin{theorem}
Let $G=D_p\times D_p$, then the spectrum of $A_\alpha(\mathcal{S}_{D_p\times D_p})$ is given as
\[
\text{Spec}(A_\alpha(\mathcal{S}_{D_p\times D_p})) =
\left(
\begin{array}{cccc}
(3p^2-2p)\alpha-1 & 4p^2\alpha-1 & (3p^2+1)\alpha-1  \\
p^2-2 & 2p^2-2p-1 & p^2+2p-1 
\end{array}
\right)
\]
 and rest of $4$ eigenvalues are given by the $4\times 4$ determinant
 \[\tiny
\left|
\begin{array}{cccc}
(4p^2-1)\alpha-\lambda & (p^2+2p)(1-\alpha) & (2p^2-2p)(1-\alpha) & (p^2-1)(1-\alpha) \\

(1-\alpha) & (p^2+2p-1)(1-\alpha)+3p^2\alpha-\lambda & (2p^2-2p)(1-\alpha) & 0\\

(1-\alpha) & (p^2+2p)(1-\alpha) & (2p^2-2p-1)(1-\alpha)+(4p^2-1)\alpha-\lambda & (p^2-1)(1-\alpha) \\

(1-\alpha) & 0 & (2p^2-2p)(1-\alpha) & [(p^2-2)+(2p^2-2p+1)\alpha]-\lambda \\
\end{array}
\right|=0
\]
\end{theorem}

\begin{proof}
Let \[
G = D_p \times D_p = \left\{
\begin{array}{l}
(a, a),\ (a, a^2),\ \dots,\ (a, a^{p-1}),\ (a, e), \\
(a, b),\ \dots,\ (a, a^{p-1}b), \\
(a^2, a),\ (a^2, a^2),\ \dots,\ (a^2, a^{p-1}),\ (a^2, e), \\
(a^2, b),\ \dots,\ (a^2, a^{p-1}b), \\
\vdots \\
(a^{p-1}, a),\ (a^{p-1}, a^2),\ \dots,\ (a^{p-1}, a^{p-1}),\ (a^{p-1}, e), \\
(a^{p-1}, b),\ \dots,\ (a^{p-1}, a^{p-1}b), \\
(e, a),\ (e, a^2),\ \dots,\ (e, a^{p-1}),\ (e, e), \\
(e, b),\ \dots,\ (e, a^{p-1}b), \\
(b, a),\ (b, a^2),\ \dots,\ (b, a^{p-1}),\ (b, e), \\
(b, b),\ \dots,\ (b, a^{p-1}b), \\
\vdots \\
(a^{p-1}b, a),\ (a^{p-1}b, a^2),\ \dots,\ (a^{p-1}b, a^{p-1}),\ (a^{p-1}b, e), \\
(a^{p-1}b, b),\ \dots,\ (a^{p-1}b, a^{p-1}b)
\end{array}
\right\}
\]

The order of $D_p\times D_p$ that is  $o(D_p\times D_p)=4p^2$, so possible orders of elements will be $1,p, p^2, 2p,4p, 2p^2, 4p^2$. Now, we know that the order of elements $(x,y)\in D_p\times D_p$ is the $\mathrm{lcm}(o(x),o(y))$. Here, every ordered pair of rotations will be of order $p$ as the order of $a^i, 1\leq i\leq p$ in $D_p$ is $p$, so $\mathrm{lcm}(o(a^i), o(a^j))=p$, where $1\leq j\leq p$. Next, on taking an ordered pair of rotation and reflection, $(a^i,a^{j}b)$ or $(a^{j}b,a^i)$  where $1\leq i\leq p$ and $1\leq j\leq p$, here we know the order of every reflection is $2$ and the order of every rotation is $2$ so $\mathrm{lcm}(o(p),o(2))=2p$ for $p\neq 2$. The next case is an ordered pair of reflections $(a^{i}b,a^{j}b)$, since every reflection is of order $2$ so $\mathrm{lcm}(o(a^{i}b),o(a^{j}b))=2$. The remaining element is the identity element, $(e,e)$, which will have order $1$. Therefore, there will be elements of order $2,p,2p$ and identity with order $1$, and there will be precisely $p^2-1$ elements of order $p$, $2p^2-2p$ elements of order $2p$, $p^2+2p$ elements of order $2$, and one element of order $1$. By the definition of the superpower graph of a finite group, we have $\mathcal{S}_{D_p\times D_p}$ the following:
\begin{figure}[h]
\centering
\includegraphics[width=0.5\linewidth]{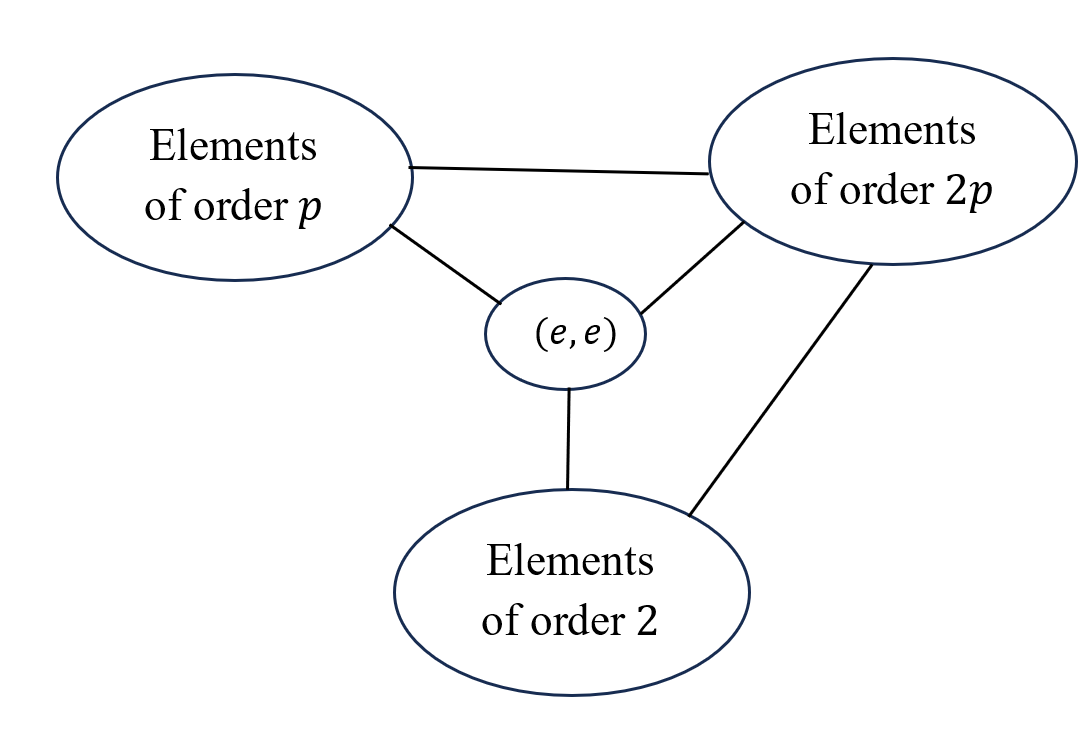}
\caption{$\mathcal{S}_{D_p\times D_p}$}
\label{fig:enter-label}
\end{figure}
Each element of order $p$ will form a clique similarly; elements of order 2 and order $2p$ will form cliques in themselves. For the adjacency matrix, on partitioning the vertices of the graph as follows: $V_1 = \{(e,e)\}, \quad V_2 = \text{set of elements of order } 2, \quad V_3 = \text{set of elements of order } 2p, \quad V_4 = \text{set of elements of order } p$. So the adjacency matrix will be given as
\[
A(\mathcal{S}_{D_{p}\times D_p}) =
\left[
\begin{array}{c|c|c|c}
O_1 & J_{1\times (p^2+2p)} & J_{1\times (2p^2-2p)} & J_{1\times (p^2-1)} \\
\hline
J_{(p^2+2p)\times1} & (J-I)_{p^2+2p} & J_{(p^2+2p)\times(2p^2-2p)} & O_{(p^2+2p)\times (p^2-1)} \\
\hline
J_{(2p^2-2p)\times1} & J_{(2p^2-2p)\times(p^2+2p)} & (J-I)_{(2p^2-2p)} & J_{(2p^2-2p)\times(p^2-1)} \\
\hline
J_{(p^2-1)\times1} & O_{(p^2-1)\times(p^2+2p)} & J_{(p^2-1)\times(2p^2-2p)} & (J-I)_{(p^2-1)} \\
\end{array}
\right]
\]
Also 
\[
\mathcal{D}(\mathcal{S}_{D_{p}\times D_p}) =
\left[
\begin{array}{c|c|c|c}
4p^2-1 & O_{1\times (p^2+2p)} & O_{1\times (2p^2-2p)} & O_{1\times (p^2-1)} \\
\hline
O_{(p^2+2p)\times1} & 3p^2I_{p^2+2p} & O_{(p^2+2p)\times(2p^2-2p)} & O_{(p^2+2p)\times (p^2-1)} \\
\hline
O_{(2p^2-2p)\times1} & O_{(2p^2-2p)\times(p^2+2p)} & (4p^2-1)I_{(2p^2-2p)} & O_{(2p^2-2p)\times(p^2-1)} \\
\hline
O_{(p^2-1)\times1} & O_{(p^2-1)\times(p^2+2p)} & O_{(p^2-1)\times(2p^2-2p)} & (3p^2-2p-1)I_{(p^2-1)} \\
\end{array}
\right]
\]
Therefore, the $A_\alpha$ adjacency will be given as 
\[\tiny
A_\alpha =
\left[
\begin{array}{c|c|c|c}
(4p^2-1)\alpha & (1-\alpha)J_{1\times (p^2+2p)} & (1-\alpha)J_{1\times (2p^2-2p)} & (1-\alpha)J_{1\times (p^2-1)} \\
\hline
(1-\alpha)J_{(p^2+2p)\times1} & [(1-\alpha)J+((3p^2+1)\alpha-1)I]_{p^2+2p} & (1-\alpha)J_{(p^2+2p)\times(2p^2-2p)} & O_{(p^2+2p)\times (p^2-1)} \\
\hline
(1-\alpha)J_{(2p^2-2p)\times1} & (1-\alpha)J_{(2p^2-2p)\times(p^2+2p)} & [(1-\alpha)J+(4p^2\alpha-1)I]_{(2p^2-2p)} & (1-\alpha)J_{(2p^2-2p)\times(p^2-1)} \\
\hline
(1-\alpha)J_{(p^2-1)\times1} & O_{(p^2-1)\times(p^2+2p)} & (1-\alpha)J_{(p^2-1)\times(2p^2-2p)} & [(1-\alpha)J+(3p^2\alpha-2p\alpha-1)I]_{(p^2-1)} \\
\end{array}
\right]
\]
Hence, the characteristic equation will be 
\[\tiny
\left|
\begin{array}{c|c|c|c}
(4p^2-1)\alpha-\lambda & (1-\alpha)J_{1\times (p^2+2p)} & (1-\alpha)J_{1\times (2p^2-2p)} & (1-\alpha)J_{1\times (p^2-1)} \\
\hline
(1-\alpha)J_{(p^2+2p)\times1} & [(1-\alpha)J+((3p^2+1)\alpha-1-\lambda)I]_{p^2+2p} & (1-\alpha)J_{(p^2+2p)\times(2p^2-2p)} & O_{(p^2+2p)\times (p^2-1)} \\
\hline
(1-\alpha)J_{(2p^2-2p)\times1} & (1-\alpha)J_{(2p^2-2p)\times(p^2+2p)} & [(1-\alpha)J+(4p^2\alpha-1-\lambda)I]_{(2p^2-2p)} & (1-\alpha)J_{(2p^2-2p)\times(p^2-1)} \\
\hline
(1-\alpha)J_{(p^2-1)\times1} & O_{(p^2-1)\times(p^2+2p)} & (1-\alpha)J_{(p^2-1)\times(2p^2-2p)} & [(1-\alpha)J+(3p^2-2p)\alpha-1-\lambda)I]_{(p^2-1)} \\
\end{array}
\right|=0
\]
Now, applying $R_{3p^2+i}\to R_{3p^2+i}-R_{3p^2+1}$ where $2\leq i\leq4p^2$. After this step, apply the column transformation$C_{3p^2+1}\to C_{3p^2+1}+C_{3p^2+2}+\cdots +C_{4p^2}$. So, we have 
$[-\lambda-(2p-3p^2)\lambda-1]^{p^2-2}=0$ and remaining part is as

\[\tiny
\left|
\begin{array}{c|c|c|c}
(4p^2-1)\alpha-\lambda & (1-\alpha)J_{1\times (p^2+2p)} & (1-\alpha)J_{1\times (2p^2-2p)} & (p^2-1)(1-\alpha) \\
\hline
(1-\alpha)J_{(p^2+2p)\times1} & [(1-\alpha)J+((3p^2+1)\alpha-1-\lambda)I]_{p^2+2p} & (1-\alpha)J_{(p^2+2p)\times(2p^2-2p)} & O_{(p^2+2p)\times 1} \\
\hline
(1-\alpha)J_{(2p^2-2p)\times1} & (1-\alpha)J_{(2p^2-2p)\times(p^2+2p)} & [(1-\alpha)J+(4p^2\alpha-1-\lambda)I]_{(2p^2-2p)} & (p^2-1)(1-\alpha)J_{(2p^2-2p)\times1} \\
\hline
(1-\alpha) & 0 & (1-\alpha)J_{(p^2-1)\times(2p^2-2p)} & [(p^2-2)+(2p^2-2p+1)\alpha]-\lambda \\
\end{array}
\right|=0
\]
Again, performing row operation and in another step column operation as $R_j\to R_j-R_2$, where $3\leq j\leq p^2+2p+1$. Next column operation as $C_2\to C_2+C_3+\cdots+C_{p^2+2p+1}$ and $C_{p^2+2p+2}\to C_{p^2+2p+2}+C_{p^2+2p+3}+\cdots+C_{3p^2+1}$. So we will have $[-\lambda+(3p^2+1)\alpha-1]^{(p^2+2p-2)}[-\lambda+4p^2\alpha-1]^{(2p^2-2p-1)}=0$ and rest of 4 roots will be given by $4\times 4$ the determinant.
\[\tiny
\left|
\begin{array}{cccc}
(4p^2-1)\alpha-\lambda & (p^2+2p)(1-\alpha) & (2p^2-2p)(1-\alpha) & (p^2-1)(1-\alpha) \\

(1-\alpha) & (p^2+2p-1)(1-\alpha)+3p^2\alpha-\lambda & (2p^2-2p)(1-\alpha) & 0\\

(1-\alpha) & (p^2+2p)(1-\alpha) & (2p^2-2p-1)(1-\alpha)+(4p^2-1)\alpha-\lambda & (p^2-1)(1-\alpha) \\

(1-\alpha) & 0 & (1-\alpha)(2p^2-2p) & [(p^2-2)+(2p^2-2p+1)\alpha]-\lambda \\
\end{array}
\right|=0
\]
\end{proof}

Using the complete description of the $A_\alpha$-adjacency spectrum, which was just completed in the previous theorem, we will now deduce adjacency eigenvalues of $\mathcal{S}_{D_p\times D_p}$ using $A_\alpha$-adjacency eigenvalues. This will provide us with information on the walk counts and cycles of $\mathcal{S}_{D_p\times D_p}$.

\begin{corollary}
Let $G=D_p\times D_p$, then the spectrum of $A(\mathcal{S}_{D_p\times D_p})$ is given as, $-1$ with algebraic multiplicity $4p^2-4$ and the remaining eigenvalues will be given by the equation 
$$\lambda^4
+  A\lambda^3
+ B\lambda^2
+ C\lambda
+ D = 0$$
where
$A=(4 - 4p^2), B=(p^4 + 2p^3 - 13p^2 - 2p + 6), C=(2p^6 + 2p^5 - 3p^4 + 4p^3 - 11p^2 - 6p + 4) $ and $D=(2p^6 + 2p^5 - 4p^4 + 2p^3 - 2p^2 - 4p + 1)$.
\end{corollary}
\begin{proof}
By the definition of $A_\alpha(G)$, we have 
$$A_\alpha(G)=\alpha D(G)+(1-\alpha)A(G)$$
For, $\alpha=0$, we have 
$$A_0(G)=A(G)$$
which is the adjacency matrix of graph $G$. In \emph{theorem 3.1} for $\alpha=0$, we have adjacency spectrum as $-1$ with algebraic multiplicity $4p^2-4$. The remaining $4$ eigenvalues will be given by solving the determinant
\[
\left|
\begin{array}{cccc}
-\lambda & (p^2+2p) & (2p^2-2p) & (p^2-1) \\

1 & (p^2+2p-1)-\lambda & (2p^2-2p) & 0\\

1 & (p^2+2p) & (2p^2-2p-1)-\lambda & (p^2-1) \\

1 & 0 & 2p^2-2p & (p^2-2)-\lambda \\
\end{array}
\right|=0
\]
Which, on solving, we will get the rest of the eigenvalues by the bi-quadratic equation below
$$\lambda^4
+  A\lambda^3
+ B\lambda^2
+ C\lambda
+ D = 0$$
where
$A=(4 - 4p^2), B=(p^4 + 2p^3 - 13p^2 - 2p + 6), C=(2p^6 + 2p^5 - 3p^4 + 4p^3 - 11p^2 - 6p + 4) $ and $D=(2p^6 + 2p^5 - 4p^4 + 2p^3 - 2p^2 - 4p + 1)$.

\end{proof}

Using the complete description of the adjacency spectrum, which was just completed in the previous theorem, we will now generalize the spectral analysis of $\mathcal{S}_{D_p\times D_p}$. This will provide us with information on the walk counts and cycles of $\mathcal{S}_{D_p\times D_p}$.  The Laplacian spectrum contains additional information, particularly on the graph's connectivity, spanning tree enumeration, and other topological features.  We investigate the graph's Laplacian spectral properties directly based on the structure of the graph as shown by its adjacency matrix.

\begin{corollary}

Let \( G = D_p \times D_p \). Then, the Laplacian spectrum of Superpower graph of $D_p\times D_p$ is given as
\[
\text{Spec}(L(\mathcal{S}_{D_p\times D_p})) =
\left(
\begin{array}{ccccc}
3p^2-2p & 4p^2 & 3p^2+1 & 2p^2-2p+1 & 0  \\
p^2-2 & 2p^2-2p+1 & p^2+2p-1 & 1 & 1
\end{array}
\right)
\]

\end{corollary}
\begin{proof}
From the above theorem, we have the adjacency matrix of the Superpower graph of $D_p\times D_p$ is given as 
\[
A(\mathcal{S}_{D_{p}\times D_p}) =
\left[
\begin{array}{c|c|c|c}
O_1 & J_{1\times (p^2+2p)} & J_{1\times (2p^2-2p)} & J_{1\times (p^2-1)} \\
\hline
J_{(p^2+2p)\times1} & (J-I)_{p^2+2p} & J_{(p^2+2p)\times(2p^2-2p)} & O_{(p^2+2p)\times (p^2-1)} \\
\hline
J_{(2p^2-2p)\times1} & J_{(2p^2-2p)\times(p^2+2p)} & (J-I)_{(2p^2-2p)} & J_{(2p^2-2p)\times(p^2-1)} \\
\hline
J_{(p^2-1)\times1} & O_{(p^2-1)\times(p^2+2p)} & J_{(p^2-1)\times(2p^2-2p)} & (J-I)_{(p^2-1)} \\
\end{array}
\right]
\]
Next, we have $A_\alpha(G)=\alpha D(G)+(1-\alpha)A(G)$, and $A_\beta(G)=\beta D(G)+(1-\beta)A(G)$, for $\alpha,\beta \in [0,1]$. Therefore, the Laplacian matrix will be given as
 $$D(G)-A(G)=\frac{A_\alpha(G)-A_\beta(G)}{\alpha-\beta}=L(G)$$
\[
L(A(\mathcal{S}_{D_{p}\times D_p})) =
\left[
\begin{array}{c|c|c|c}
4p^2-1 & -J_{1\times (p^2+2p)} & -J_{1\times (2p^2-2p)} & -J_{1\times (p^2-1)} \\
\hline
-J_{(p^2+2p)\times1} & -(J-3p^2I)_{p^2+2p} & -J_{(p^2+2p)\times(2p^2-2p)} & O_{(p^2+2p)\times (p^2-1)} \\
\hline
-J_{(2p^2-2p)\times1} & -J_{(2p^2-2p)\times(p^2+2p)} & -(J-(4p^2-1)I)_{(2p^2-2p)} & -J_{(2p^2-2p)\times(p^2-1)} \\
\hline
-J_{(p^2-1)\times1} & O_{(p^2-1)\times(p^2+2p)} & -J_{(p^2-1)\times(2p^2-2p)} & -(J-3p^2I)_{(p^2-1)} \\
\end{array}
\right]
\]
Now, applying the same row and column transformation used in Theorem $3.2$, we have eigenvalues with their respective multiplicities as
\[
\text{Spec}(L(\mathcal{S}_{D_p\times D_p})) =
\left(
\begin{array}{ccccc}
3p^2-2p & 4p^2 & 3p^2+1 & 2p^2-2p+1 & 0  \\
p^2-2 & 2p^2-2p+1 & p^2+2p-1 & 1 & 1
\end{array}
\right)
\]
\end{proof}

Prior to introducing the \emph{signless Laplacian spectrum}, we remind that this kind of a spectrum gives a complementary insight of the Laplacian spectrum.
Although the Laplacian matrix of a graph is defined as $L(G) = D(G) - A(G)$, the signless Laplacian is defined as $Q(G) = D(G) + A(G)$.
The spectrum of the signless Laplacian. All the eigenvalues of \(Q(G)\) are referred to as the signless Laplacian eigenvalues and form the signless Laplacian spectrum of the graph.
The spectrum is especially useful in the study of graph connectivity, spectral bounds, and bipartiteness.
More precisely, the Laplacian and signless Laplacian have identical spectra on a bipartite graph, but on other graphs, the signless Laplacian has a very different and helpful spectral characterization.

\begin{corollary}
For the group $G=D_p\times D_p$, the signless Laplacian spectrum of $\mathcal{S}_{D_p\times D_p}$ is given as
\[
\text{Spec}(Q(\mathcal{S}_{D_p\times D_p})=
\left(
\begin{array}{cccc}
3p^2-2p-2 & 4p^2-2 & 3p^2-1  \\
p^2-2 & 2p^2-2p-1 & p^2+2p-1 
\end{array}
\right)
\]
 and rest of $4$ eigenvalues are given by the $4\times 4$ determinant and each of these $4$ eigenvalues will be multiplied by $2$.
 \[\tiny
\left|
\begin{array}{cccc}
(4p^2-1)\alpha-\lambda & (p^2+2p)(1-\alpha) & (2p^2-2p)(1-\alpha) & (p^2-1)(1-\alpha) \\

(1-\alpha) & (p^2+2p-1)(1-\alpha)+3p^2\alpha-\lambda & (2p^2-2p)(1-\alpha) & 0\\

(1-\alpha) & (p^2+2p)(1-\alpha) & (2p^2-2p-1)(1-\alpha)+(4p^2-1)\alpha-\lambda & (p^2-1)(1-\alpha) \\

(1-\alpha) & 0 & (1-\alpha)(2p^2-2p) & [(p^2-2)+(2p^2-2p+1)\alpha]-\lambda \\
\end{array}
\right|=0
\]
\end{corollary}

\begin{proof}
For the case the of signless Laplacian, since $A_\alpha(G)=\alpha D(G)+(1-\alpha)A(G)$. Putting $\alpha=\frac{1}{2}$, we have 
$$2A_{\frac{1}{2}}(G)=D(G)+A(G)$$
Therefore, the sigenless Laplacian eigenvalues will be 
\[
\text{Spec}(2A_\alpha(\mathcal{S}_{D_p\times D_p}))=\text{Spec}(Q(\mathcal{S}_{D_p\times D_p})) =
\left(
\begin{array}{cccc}
\frac{(3p^2-2p)}{2}-1 & 2p^2-1 & \frac{(3p^2+1)}{2}-1  \\
p^2-2 & 2p^2-2p-1 & p^2+2p-1 
\end{array}
\right)
\]
The other remaining eigenvalues will be given by the determinant below and each of these $4$ eigenvalues will be multiplied by $2$. 
\[
\left|
\begin{array}{cccc}
\frac{(4p^2-1)}{2}-\lambda & \frac{(p^2+2p)}{2} & (p^2-p) & \frac{(p^2-1)}{2} \\

\frac{1}{2} & \frac{(4p^2+2p-1)}{2}-\lambda & (p^2-p) & 0\\

\frac{1}{2} & \frac{(p^2+2p)}{2} & (3p^2-p-1)-\lambda & \frac{(p^2-1)}{2} \\

\frac{1}{2} & 0 & p^2-p & \frac{(4p^2-2p-3)}{2}-\lambda \\
\end{array}
\right|=0
\]
\end{proof}

Our focus now changes to a more general and wide scenario after a thorough examination of the Superpower graph of $D_p \times D_p$, including its spectrum and Laplacian spectrum.  In order to get a deeper understanding of these graph topologies, we naturally extend our emphasis to the Superpower graph of $D_{p^k}$.  In particular, we will now examine this more extended graph's $A_\alpha$-adjacency matrix.  Finding out how the spectral characteristics we saw in the $D_p \times D_p$ instance translate and change in this more intricate and wider context is the objective of this next section of our research.

\section{$A_\alpha$-Adjacency of Superpower Graph of $D_{p^k}$}
\begin{theorem}
Let $G=D_{p^k}$ for $p$ being odd prime; then the $A_\alpha$-adjacency spectrum of $\mathcal{S}_{D_{p^k}}$ is given as 
$$[(p^k+1)\alpha-1]^{(p^k-1)}, [p^k\alpha-1]^{(p^k-2)}$$ and the rest of the three eigenvalues are given by the $3\times3$ determinant below:
\[
\left|
\begin{array}{ccc}
(2p^k-1)\alpha-\lambda & (p^k-1)(1-\alpha) & p^k(1-\alpha)  \\
(1-\alpha) & p^k+\alpha-2-\lambda & 0  \\
(1-\alpha) & 0 & p^k+\alpha-1-\lambda  \\
\end{array}
\right|=0
\]

\end{theorem}

\begin{proof}  
It can be easily observed that the adjacency matrix of the supergraph of $D_{p^k}$ is 
\[
A(\mathcal{S}_{D_{p^k}}) =
\left[
\begin{array}{c|c|c}
0 & J_{1\times (p^k-1)} & J_{1\times p^k} \\
\hline
J_{(p^k-1)\times1} & (J-I)_{p^k-1} & O_{(p^k-1)\times(p^k)}\\
\hline
J_{p^k\times1} & O_{p^k\times(p^k-1)} & (J-I)_{p^k} \\
\end{array}
\right]
\]
Also, the diagonal matrix is \[
\mathcal{D}(\mathcal{S}_{D_{p^k}}) =
\left[
\begin{array}{c|c|c}
(2p^k-1) & O_{1\times (p^k-1)} & O_{1\times p^k} \\
\hline
O_{(p^k-1)\times1} & (p^k-1)I_{p^k-1} & O_{(p^k-1)\times(p^k)}\\
\hline
O_{p^k\times1} & O_{p^k\times(p^k-1)} & (p^k)I_{p^k} \\
\end{array}
\right]
\]
Therefore, characteristic equation of $A_{\alpha}(\mathcal{S}_{D_{p^k}})$ will be given as 
\[
\left|
\begin{array}{c|c|c}
(2p^k-1)\alpha-\lambda & (1-\alpha)J_{1\times (p^k-1)} & (1-\alpha)J_{1\times p^k} \\
\hline
(1-\alpha)J_{(p^k-1)\times1} & [(1-\alpha)J+(p^k\alpha-1-\lambda)I]_{p^k-1} & O_{(p^k-1)\times(p^k)}\\
\hline
(1-\alpha)J_{p^k\times1} & O_{p^k\times(p^k-1)} & [(1-\alpha)J+((p^k+1)\alpha-1-\lambda)I]_{p^k} \\
\end{array}
\right|=0
\]
Now, applying $R_{p^k+2}-R_{p^k+1}, R_{p^k+3}-R_{p^k+1},\cdots , R_{2p^k+2}-R_{p^k+1}$. After this, apply column transformation as $C_{p^k+1}\to C_{p^k+1}+C_{p^k+2}+\cdots +C_{2p^k}$, we will get 
\[
\left|
\begin{array}{cc}
A & B \\
C & D \\
\end{array}
\right|=0
\]
Where $A=\left|
\begin{array}{cc}
A & B \\
O & D \\
\end{array}
\right|=0$
Here, it is easy to calculate the determinant of $D$, and it is 
$$[-\lambda+(p^k+1)\alpha-1]^{p^k-1}=0$$
Next, we have to calculate the determinant of $D$.
\[|D|=
\left|
\begin{array}{cccccc}
(2p^k-1)\alpha-\lambda & 1-\alpha & 1-\alpha & \cdots & 1-\alpha & p^k(1-\alpha)  \\
(1-\alpha) & (p^k-1)\alpha-\lambda & 1-\alpha & \cdots & 1-\alpha & 0  \\
0 & 1-p^k\alpha+\lambda & p^k\alpha-1-\lambda & \cdots & 0 & 0  \\
\vdots & \vdots & \vdots & \ddots & \vdots & \vdots  \\
0 & 1-p^k\alpha+\lambda & 0 & \cdots & p^k\alpha-1-\lambda & 0  \\
(1-\alpha) & 0 & 0 & \cdots & 0 & p^k\alpha-\lambda  \\
\end{array}
\right|
\]
Now, applying row transformation as $R_3-R_2, R_4-R_2,\cdots, R_{p^k}-R_2$, we have 
\[|D|=
\left|
\begin{array}{cccccc}
(2p^k-1)\alpha-\lambda & 1-\alpha & 1-\alpha & \cdots & 1-\alpha & p^k(1-\alpha)  \\
(1-\alpha) & (p^k-1)\alpha-\lambda & 1-\alpha & \cdots & 1-\alpha & 0  \\
(1-\alpha) & (1-\alpha) & (p^k-1)\alpha-\lambda & \cdots & 1-\alpha & 0  \\
\vdots & \vdots & \vdots & \ddots & \vdots & \vdots  \\
1-\alpha & 1-\alpha & 1-\alpha & \cdots & (p^k-1)\alpha-\lambda & 0  \\
(1-\alpha) & 0 & 0 & \cdots & 0 & p^k\alpha-\lambda  \\
\end{array}
\right|
\]
Again, on applying $C_2\to C_2+C_3+\cdots+C_{p^k}$, we will get 
\[|D|=[-\lambda+(p^k\alpha-1)]^{p^k-2}
\left|
\begin{array}{ccc}
(2p^k-1)\alpha-\lambda & 1-\alpha & p^k(1-\alpha)  \\
(1-\alpha) & p^k+\alpha-2-\lambda & 0  \\
(1-\alpha) & 0 & p^k+\alpha-1-\lambda  \\
\end{array}
\right|
\]
Therefore, the eigenvalues of $A_{\alpha}(\mathcal{S}_{D_{p^k}})$ are $(p^k+1)\alpha-1$ with algebraic multiplicity $p^k-1$, $p^k\alpha-1$ with algebraic multiplicity $p^k-2$ and rest of the three eigenvalues are given by the determinant below:
\[|D|=
\left|
\begin{array}{ccc}
(2p^k-1)\alpha-\lambda & (p^k-1)(1-\alpha) & p^k(1-\alpha)  \\
(1-\alpha) & p^k+\alpha-2-\lambda & 0  \\
(1-\alpha) & 0 & p^k+\alpha-1-\lambda  \\
\end{array}
\right|=0
\]
Or by solving this, we have a cubic equation as below

\[\small
\begin{aligned}
&(-2p^k + 1)\lambda^3 + \left(4p^{2k} - 8p^k + 3 + (4p^k - 1)\alpha\right)\lambda^2 \\
&+ \left(-2p^{3k} + 11p^{2k} - 11p^k + 3 + (4p^{2k} - 6p^k)\alpha^2 + (-12p^{2k} + 14p^k - 2)\alpha\right)\lambda \\
&+ \left(-4p^{3k} + 10p^{2k} - 6p^k + (4p^k - 4p^{2k})\alpha^3 + (-4p^{3k} + 18p^{2k} - 12p^k)\alpha^2 + (8p^{3k} - 23p^{2k} + 13p^k - 1)\alpha + 1\right)=0
\end{aligned}
\]
\end{proof}
\section{Conclusion}

In this paper, we have investigated the spectral properties of the superpower graphs of the direct product of two dihedral groups, particularly on the group $G = D_p \times D_p$, where $D_p$ is the dihedral group of order $2p$ with $p\neq 2$. Also, we computed the characteristic polynomial and determined the adjacency and Laplacian spectra of the superpower graphs $\mathcal{S}_G$, thereby contributing to the broader understanding of the structural and spectral behavior of such graphs.

Furthermore, we extended our investigation to the $A_\alpha$-adjacency spectrum of superpower graphs of finite groups, analyzing both $\mathcal{S}_{D_{p^k}}$ and $\mathcal{S}_{D_p \times D_p}$ for $p\neq 2$. This has been shown that investigating the $A_\alpha$-matrix, which interpolates between the adjacency and signless Laplacian matrices, provides a more comprehensive understanding of spectrum analysis and can act as a connection between various spectral characteristics.

These observations are useful in the context of algebraic and spectral graph theory, as it is related to group-based graphs, and they also present various opportunities for further development. To understand further, one may specifically examine various families of non-abelian groups and corresponding graph-theoretic constructions, such as the symmetric groups, di-cyclic groups, or even bigger direct products.  Moreover, some new structural or spectral bound characterizations might be obtained by a detailed study of the relationships between spectral properties and structural features.

\section*{Acknowledgement}

\section*{Declarations}
\textbf{Conflict of interest:} The Authors claim to have no conflicts of interest.


\end{document}